\documentclass[12pt,reqno]{amsart}
\usepackage{amssymb,a4wide,eucal,amscd,yhmath,amsgen,amsopn}
\usepackage{fancyhdr}
\usepackage{amsmath,hyperref,xcolor}
\usepackage{mathrsfs}
\usepackage[all,cmtip]{xy}
\SilentMatrices
\SelectTips{eu}{}
\newtheorem{theorem}{Theorem}[section]
\newtheorem{proposition}[theorem]{Proposition}

\newtheorem{lemma}[theorem]{Lemma}

\newtheorem{rem}[theorem]{Remark}
\theoremstyle{definition}
\newtheorem{definition}[theorem]{Definition}

\newcommand\projective\mathbf
\newcommand\PP{\projective P}

\newcommand\OO{\mathcal O}

\newcommand\CC{\mathbb C}

\newcommand\onto\twoheadrightarrow
\newcommand\lra\longrightarrow
\newcommand\dar\downarrow

\begin{document}

\title{Cohomological splitting conditions of vector bundles on $\PP^{n_1}\times\cdots\times\PP^{n_s}$}
\author{Damian M Maingi}
\date{September, 2024}

\address{Department of Mathematics\\Sultan Qaboos University\\ P.O Box 50, 123 Muscat, Oman\\
Department of Mathematics\\University of Nairobi\\P.O Box 30197, 00100 Nairobi, Kenya\\https://orcid.org/0000-0001-9267-9388}
\email{dmaingi@squ.edu.om, dmaingi@uonbi.ac.ke} 

\begin{abstract}
In this paper we extend the results of Ballico and Malaspina on regularity and splitting  conditions on multiprojective spaces 
$X=\PP^{n_1}\times\ldots\times\PP^{n_s}$.
\end{abstract}

\maketitle

\section{Introduction}
\noindent The concept of regularity of $\mathscr{E}$ a coherent sheaf on $\mathbb{P}^n$ which was introduced by Mumford \cite{12} turned out
to be a key tool in algebraic geometry and commutative algebra in that it measures the complexity of a sheaf or a module.
The regularity of a coherent sheaf $\mathscr{E}$ is the smallest twist $r$ for which the sheaf is globally generated.
It plays a very important role in a number of classic areas for instance muduli problems, vanishing theorems, syzygies, linear systems etc. 
After almost 60 years many authors have contributed greatly see \cite{1,2,3,4,5,6,7} and \cite{10,11,12,13} and it still remains a fertile ground for research.


\noindent The notion of regularity on a multiprojective space $\PP^{n_1}\times\cdots\times\PP^{n_s}$ was delt with by Ballico and Malaspina \cite{4}
in this paper we extend the splitting criteria for vector bundles on a biprojective space $\PP^{n}\times\PP^{m}$  to a multiprojective space $\PP^{n_1}\times\cdots\times\PP^{n_s}$.

\begin{definition} A coherent sheaf F on $\PP^{n_1}\times\cdots\times\PP^{n_s}$ is said to be $(p_1,\ldots,p_s)$-regular
if, for all $i>0$,
\[H^i(F(p_1,\ldots,p_s)\otimes\OO(k_1,\ldots, k_s)) = 0\]
whenever $k_1 + \cdots+k_s = -i$ and $-n_j \leq k_j \leq 0$ for any $j = 1,\ldots,s$.
\end{definition}

\begin{theorem}[K\"{u}nneth formula]
 Let $X$ and $Y$ be projective varieties over a field $k$. 
 Let $\mathscr{F}$ and $\mathscr{G}$ be coherent sheaves on $X$ and $Y$ respectively.
 Let $\mathscr{F}\boxtimes\mathscr{G}$ denote $p_1^*(\mathscr{F})\otimes p_2^*(\mathscr{G})$\\
 then $\displaystyle{H^m(X\times Y,\mathscr{F}\boxtimes\mathscr{G}) \cong \bigoplus_{p+q=m} H^p(X,\mathscr{F})\otimes H^q(Y,\mathscr{G})}$.
\end{theorem}

\begin{lemma}
Let $X=\PP^{n_s}\times\cdots\times\PP^{n_s}$ then\\
$\displaystyle{H^t(X,\OO_X (a_1,\cdots,a_s))\cong \bigoplus_{t=\sum_{i=1}^{s}{t_i}} H^{t_1}(\PP^{n_1},\OO_{\PP^{n_1}}(a_1))\otimes \cdots \otimes H^{t_s}(\PP^{n_s},\OO_{\PP^{n_s}}(a_s))}$
 \end{lemma}

\begin{rem}
\begin{enumerate}
 \item We will often say “regular” instead of “ $(0,\cdots,0)$-regular ”, and “ $p$-regular ” instead of “ $(p,\cdots, p)$-regular ”. 
We define the regularity of $F$, $Reg(F)$, as the least integer $p$ such that $F$ is $p$-regular. We set $Reg(\mathscr{F}) = -\infty$ if there is no such integer.
\item The K\"{u}nneth's formula tells that $\OO(a_1,\cdots,a_s)$ is regular if and only if $a_i\geq0$ , $i=1,\cdots,s$.
In fact \[H^{n_1+\cdots+n_s}(\OO(a_1-n_1,\cdots,a_s-n_s)) = H^{n_1}(\OO(a_1-n_1))\otimes\cdots\otimes H^{n_s}(\OO(a_s-n_s)) = 0\]
if and only if $a_i\geq0$ , $i=1,\cdots,s$.\\
Since
\[H^{n_1}(\OO(a_1-n_1, a_2, \cdots, a_s)) \cong H^{n_1}(\OO(a_1-n_1))\otimes H^0(\OO(a_2))\otimes\cdots\otimes H^0(\OO(a_s))\]
\[H^{n_2}(\OO(a_1, a_2-n_2, \cdots, a_s)) \cong H^0(\OO(a_1))\otimes H^{n_2}(\OO(a_2-n_2))\otimes\cdots\otimes H^0(\OO(a_s))\]
$\cdots\cdots\cdots\cdots\cdots\cdots\cdots\cdots\cdots\cdots\cdots\cdots\cdots\cdots\cdots\cdots\cdots\cdots\cdots$
\[H^{n_s}(\OO(a_1, a_2, \cdots, a_s-n_s)) \cong H^0(\OO(a_1))\otimes H^0(\OO(a_2))\otimes\cdots\otimes H^{n_s}(\OO(a_s-n_s))\]
we see
that, if $\OO(a_1,\cdots,a_s)$ is regular, then we have $a_1,\cdots,a_s\geq0$.
Specifically $\OO$ is regular while $\OO(-1,\cdots,-1)$ is not and so $Reg(\OO) = 0$
\end{enumerate}
\end{rem}

\section{Splitting Criteria for Vector Bundles on Multiprojective spaces}

\noindent We extend the notion of regularity to $X = \PP^{n_1}\times\cdots\times\PP^{n_s}$ in order to prove some splitting criteria of vector bundles.
We need the following definition:

\begin{definition} A vector bundle $E$ on $X$ is arithmetically Cohen-Macaulay (aCM) if, for any $\displaystyle{0<i<\sum_{j=i}^sn_j}$
and for any integer $t$, $H^i(E(t,\cdots,t)) = 0$.
\end{definition}

\begin{proposition}
\begin{enumerate}
 \item The K\"{u}nneth's formula gives: $\OO(a_1,\cdots,a_s)$ is aCM if and only if $a_i-a_j\geq-n_i$ and
$a_j-a_i\geq-n_j$ for $i=1,\cdots,s$.
\item On $X$ we have the following Koszul complexes:
\[K_1:0 \rightarrow \OO(-n_1-1,\cdots,-n_s-1)\rightarrow \OO(-n_1,\cdots,-n_s-1)^{{n_1+1}\choose{n_1}}\rightarrow\cdots\rightarrow
\OO(0,-n_2-1,\cdots-n_s-1)\rightarrow0\]
\[K_2:0 \rightarrow \OO(0,-n_2-1,\cdots-n_s-1)\rightarrow\OO(0,0,\cdots,-n_s-1)\rightarrow \OO(0,0,\cdots,-n_s)^{{n_s+1}\choose{n_s}}\rightarrow\cdots\rightarrow\OO\rightarrow0\]
\[K_3:0 \rightarrow \OO(-n_1-1,\cdots,-n_s-1)\rightarrow \OO(-n_1,\cdots,-n_s-1)^{{n_1+1}\choose{n_1}}\rightarrow\cdots
\rightarrow\OO(0,\cdots-n_s-1)\rightarrow\cdots\rightarrow\OO\rightarrow0\]

\end{enumerate}

\noindent On taking cohomology of the the above exact sequences we get the isomorphisms:
\[H^{n_s}(\OO(0,0,\cdots,-n_s-1))\cong H^{n_1+\cdots+n_s}(\OO(-n_1-1,\cdots,-n_s-1))\]
\[H^{n_j}(\OO(0,\cdots,0,-n_j-1,0,\cdots,0))\cong H^{n_1+\cdots+n_s}(\OO(-n_1-1,\cdots,-n_s-1))\]
\[H^0(\OO)\cong H^{n_s}(\OO(0\cdots,-n_s-1))\]
\[H^0(\OO)\cong H^{n_j}(\OO(0,\cdots,0,-n_j-1,0\cdots,0))\]
\end{proposition}

\noindent The following results were proved by Ballico and Malaspina \cite{4} for a biprojective space $\PP^{n}\times\PP^{m}$
which we extend to a multi projective space $\PP^{n_1}\times\cdots\times\PP^{n_s}$ in themain results section.

\begin{lemma}[Theorem 1.3 \cite{3}]
 Let $E$ be a rank $r$ vector bundle on $\PP^n\times\PP^m$. Then the following conditions
are equivalent:
  \begin{enumerate}
    \renewcommand{\theenumi}{\alph{enumi}}
    \item for any $i = 1, \cdots , m + n-1$ and for any integer $t$, 
    \[H^i(E(t, t)\otimes \OO(j, k)) = 0 \]
whenever $j + k = -i$, $-n\leq j \leq0$ and $-m\leq k \leq 0$.
\item There are $r$ integers $t_1,\cdots, t_r$ such that $\displaystyle{E=\bigoplus_{i=1}^r\OO(t_i, t_i)}$.
\end{enumerate}
\end{lemma}

\begin{lemma}[Theorem 1.4 \cite{3}]
Let $E$ be a vector bundle on $\PP^n\times\PP^m$. Then the following conditions are equivalent:
\begin{enumerate}
    \renewcommand{\theenumi}{\alph{enumi}}
    \item for any $i = 1,\cdots, m+n-1$ and for any integer $t$, 
\[H^i(E(t, t)\otimes \OO(j, k)) = 0 \]
whenever $j + k = -i$, $-n\leq j \leq0$ and $-m\leq k \leq 0$ but $(j, k) \neq (-n, 0),(0, -m)$.
\item $E$ is a direct sum of the line bundles $\OO, \OO(0, 1)$ and $\OO(1, 0)$ with some balanced twist $(t, t)$.
\end{enumerate}
\end{lemma}

\begin{lemma} [Theorem 3.5, \cite{3}]
Let $E$ be a vector bundle on $\PP^n\times\PP^m$ with $Reg(E) = 0$. Then the following conditions are equivalent:
\begin{enumerate}
    \renewcommand{\theenumi}{\alph{enumi}}
    \item for any $i = 1,\ldots , \min(r,m + n)-1$ and for any integer $t$, 
      \[H^i(E(-1, -1)\otimes\OO(j, k)) = 0\]			
      whenever $j+k\geq -i$, $-n< j\leq 0$ and $-m < k \leq 0$.
	\\
      Moreover for any $u = 1,\ldots , n - 1$, $H^{n+u}(E(-1, -1)\otimes\OO(-n, -u -1)) = 0 $ and for
      any $v = 1, \dots , m - 1$, $H^{m+v}(E(-1,-1)\otimes\OO(-v-1,-m)) = 0$.
    \item $E$ has one of the following bundles as a direct summand: $\OO$, $\OO(0, 1)$, $\OO(1, 0)$, $\OO\boxtimes\Omega^a_{\PP^m}(a+1)$
      (where $1\leq a \leq m-1)$ or $\Omega^a_{\PP^n}(a + 1)\boxtimes\OO$ (where $1\leq a\leq n-1$).
 \end{enumerate}
 
\end{lemma}

\section{Main Results}

Generalization of results in \cite{3} to $\PP^{n_1}\times\cdots\times\PP^{n_s}$.

\begin{theorem}
 Let $E$ be a rank $r$ vector bundle on $\PP^{n_1}\times\cdots\times\PP^{n_s}$. Then the following conditions
are equivalent:
  \begin{enumerate}
    \renewcommand{\theenumi}{\alph{enumi}}
    \item for any $i = 1,\ldots,d-1$ with ${d=\displaystyle\sum_{i=1}^sn_i}$ and for any integer $t$, 
    \[H^i(E(t,\cdots,t)\otimes \OO(k_1,\dots,k_s)) = 0 \]
      whenever $d = -i$, $-n_j\leq k_j \leq0$.
\item There are $r$ integers $t_1,\ldots,t_r$ such that $\displaystyle{E\cong\bigoplus_{i=1}^r\OO(t_i,\cdots,t_i)}$.
\end{enumerate}
\end{theorem}

\begin{proof}\cite{3}


\end{proof}

\begin{theorem} 
Let $E$ be a vector bundle on $\PP^{n_1}\times\cdots\times\PP^{n_s}$. Then the following conditions are equivalent:
\begin{enumerate}
    \renewcommand{\theenumi}{\alph{enumi}}
    \item for any $i = 1,\ldots,d-1$ with ${d=\displaystyle\sum_{i=1}^sn_i}$ and for any integer $t$, 
      $H^i(E(t,\cdots,t)\otimes \OO(k_1,\dots,k_s)) = 0$ whenever $d = -i$, $-n_j\leq k_j \leq0$ \\
      but $(k_1,\cdots,k_s) \neq (-n_1,0,\cdots,0),(0,-n_2,\cdots,0),\ldots,(0,\cdots,0,-n_s)$.
    \item $E$ is a direct sum of the line bundles $\OO, \OO(0,\cdots,0, 1)$ , $\OO(0,\cdots,1, 0)$, $\dots$, and $\OO(1,0,\cdots, 0)$ with some balanced twist $(t,\cdots,t)$.
\end{enumerate}
\end{theorem}

\begin{proof}\cite{3}
\end{proof}

\begin{theorem}
 Let $E$ be a rank $r$ vector bundle on $X=\PP^{n_1}\times\cdots\times\PP^{n_s}$ and let ${d=\displaystyle\sum_{i=1}^sn_i}$ with $Reg(E)=0$.
Then the following conditions are equivalent:
\begin{enumerate}
\item For any $i=1,\cdots,\min(r,d)-1$,
\[H^i(E(-1,\cdots,-1)\otimes\OO(k_1,\cdots,k_s))=0\] whenever $d\geq-i$, $-n_j\leq k_j\leq0$ for $j=1,\cdots,s$.\\
\\
Moreover for any $a_j=1,\cdots,n_{j-1}$
\begin{align*}
H^{a_1+n_2+n_3+\cdots+n_{s-1}+n_s}(E(-1,\cdots,-1)\otimes\OO(-a_1-1,-n_2,\cdots,-n_s))=0\\
H^{n_1+a_2+n_3+\cdots+n_{s-1}+n_s}(E(-1,\cdots,-1)\otimes\OO(-n_1,-a_2-1,\cdots,-n_s))=0\\
\cdots\cdots\cdots\cdots\cdots\cdots\cdots\cdots\cdots\cdots\cdots\cdots\cdots\cdots\cdots\cdots\cdots\cdots\cdots\\
H^{n_1+n_2+\cdots+n_{s-1}+a_s}(E(-1,\cdots,-1)\otimes\OO(-n_1,\cdots,-n_{s-1},-a_{s-1}))=0
\end{align*}

\item $E$ has one of the following bundles as a direct summand:
    \begin{enumerate}
    \renewcommand{\theenumi}{\alph{enumi}}
    \item $\OO$
    \item $\OO(0,\cdots,0,1)$, $\OO(0,\cdots,1,0)$,$\ldots$, $\OO(0,1,\cdots,0,0)$ and $\OO(1,0,\cdots,0,0)$
    \item $\OO\boxtimes\Omega_{\PP^{n_i}}^{a_i}(a_i+1)$ where $1\leq a_i\leq n_i-1$ or
    \item $\Omega_{\PP^{n_i}}^{a_i}(a_i+1)\boxtimes\OO$ where $1\leq a_i\leq n_i-1$ 
    \end{enumerate}
    
\end{enumerate}

\end{theorem}

\begin{proof}
We prove (1)$\Longrightarrow$(2).\\
Since $Reg(E)=0$ then $E$ is regular but $E(-1,\ldots,-1)$ is not.\\
Since $E$ is globally generated i.e. a regular coherent sheaf is globally generated and the tensor product of a vector bundle spanned by an ample
vector bundle is ample and so we have
$a_1,\cdots,a_s>0$ $\Longrightarrow$  $E(a_1,\cdots,a_s)$ is ample.\\
Suppose $r<d$ by le Potier's vanishing theorem 
\[H^i(E^{\vee}(a_1,\cdots,a_s))=0\] for every $a_1,\cdots,a_s>0$ and $0<i<d$.\\
So by Serre's duality
\[H^i(E(-n_1-1+a_1,\cdots,-n_s-1+a_s))=0\] for every $a_1,\cdots,a_s>0$ and $0<i<d$.\\
\noindent In particular for any $0<j<d$ 
\[H^i(E(-1,\cdots,-1)\otimes\OO(k_1,\cdots,k_s))=0\] whenever $d\geq-i$, $-n_j\leq k_j\leq0$ for $j=1,\cdots,s$.\\
We say that $E(-1,\cdots,-1)$ is not regular if and only if there exists an integer $a_j\in\{0,\cdots,n_s\}$ such that
\begin{align*}
H^{n_2+n_3\cdots+n_s+a_1}(E(-1,\cdots,-1)\otimes\OO(-a_1,-n_2,\cdots,-n_s))\neq0\\
H^{n_1+n_3+\cdots+n_s+a_2}(E(-1,\cdots,-1)\otimes\OO(-n_1,-a_2,\cdots,-n_s))\neq0\\
\cdots\cdots\cdots\cdots\cdots\cdots\cdots\cdots\cdots\cdots\cdots\cdots\cdots\cdots\cdots\cdots\cdots\cdots\cdots\\
H^{n_1+n_2+\cdots+n_{s-1}+a_s}(E(-1,\cdots,-1)\otimes\OO(-n_1,\cdots,-n_{s-1},-a_s))\neq0
\end{align*}

If $a_j=0$ or $a_j=n_j$ then by a generalized version of Thm 1.4 (Ballico-Malaspina) we have direct summands 
$\OO,\OO(0,\cdots,0,1)$, $\OO(0,\cdots,1,0)$,$\ldots$, $\OO(0,1,\cdots,0,0)$ and $\OO(1,0,\cdots,0,0)$.\\
Now fix $a_j\in\{1,\cdots,n_s-1\}$ and assume\\
$H^{n_1+\cdots+a_j+\cdots+n_s}(E(-1,\cdots,-1)\otimes\OO(-n_1,\cdots,-a_j,\cdots,-n_s))\neq0$ and consider the exact sequence
\[K_1:0 \rightarrow \OO(-n_1-1,\cdots,-n_s-1)\rightarrow \OO(-n_1,\cdots,-n_s-1)^{{n_1+1}\choose{n_1}}\rightarrow\cdots\rightarrow
\OO(0,-n_2-1,\cdots-n_s-1)\rightarrow0\]
twisted by $\OO(0,\cdots,0,-1-a_j)$ that is
\[0 \rightarrow \OO(-n_1-1,\cdots,-n_s-a_j-2)\rightarrow \OO(-n_1,\cdots,-n_s-a_j-2)^{{n_1+1}\choose{n_1}}\rightarrow\]
\[\cdots\rightarrow\OO(0,-n_2-1,\cdots-n_s-a_j-2)\rightarrow0\]
and the dual of 
$0\rightarrow\OO\boxtimes\Omega_{\PP^{n_j}}^{a_j}(a_j+1)\rightarrow\OO(0,\cdots,0,1)^{{n_j+1}\choose{n_j}}\rightarrow$
\[\rightarrow\OO(0,\cdots,0,2)^{{n_j+1}\choose{n_j}}\rightarrow\cdots\rightarrow\OO(0,\cdots,0,a_j)^{{n_j+1}\choose{n_j}}\rightarrow\OO(0,\cdots,0,a_j+1)^{{n_j+1}\choose{n_j}}\rightarrow0\]
Tensoring by $E$ we get
\[0 \rightarrow \OO(-n_1-1,\cdots,-n_s-1,-1-a_j)\otimes{E}\rightarrow\OO(-n_1,-n_2-1,\cdots,-1-a_j)^{{n_1+1}\choose{n_1}}\otimes{E}\rightarrow\cdots\]
\[\rightarrow\OO(-1,\cdots,-1-a_j)^{{n_1+1}\choose{n_1}}\otimes{E}\rightarrow\OO(0,\cdots,-a_j)^{{n_1+1}\choose{n_1}}\otimes{E}\rightarrow\]
\[\rightarrow\cdots\rightarrow\OO\boxtimes\left(\Omega_{\PP^{n_j}}^{a_j}\right)^\vee(-1-a_j)\otimes E\]
Since 

\[H^{n_1+\cdots+a_j+\cdots+n_s\cdots+n_s}(E(-n_1,\cdots,-n_{s-1},-a_j)=\]\\
\[H^{1+\cdots+a_j+\cdots+1\cdots+1}(E(-1,\cdots,-1,-a_j)=\]\\
$\cdots\cdots\cdots\cdots\cdots\cdots\cdots\cdots\cdots\cdots\cdots\cdots\cdots\cdots\cdots\cdots\cdots\cdots\cdots$\\
\[H^{a_j}(E(0,\cdots,0,-a_j)=0\]
we have the surjection map $\psi_1$
\[H^0(\OO\boxtimes\left(\Omega_{\PP^{n_j}}^{a_j}\right)^\vee(-1-a_j)\otimes E)\]
\[
\begin{CD}
@.@V^{surjects}VV\\
@.H^{n_1+\cdots+a_j+\cdots+n_s\cdots+n_s}(E(-n_1,\cdots,-n_{s-1},-a_j)\\
\end{CD}
\]
We also have that 
\[H^{n_1+\cdots+a_j+\cdots+n_s}(E(-n_1-1,\cdots,-n_{s-1},-a_j-1))\cong H^{n_1+\cdots+n_s-a_j}(E^\vee(0,\cdots,-n_s-1+a_j))\]
Consider the exact sequence
\[\OO(0,\cdots,0,-n_j+a_j)\rightarrow\OO(0,\cdots,0,-n_j+a_j+1)^{{n_j+1}\choose{n_j}}\rightarrow\cdots\rightarrow
\OO^{{n_j+1}\choose{a_j}}\rightarrow\OO\boxtimes\left(\Omega_{\PP^{n_j}}^{a_j}\right)(a_j+1)\rightarrow0\]
Tensor it by $E^\vee$ to get
\[E^\vee(0,\cdots,0,-n_j+a_j)\longrightarrow E^\vee(0,\cdots,0,-n_j+a_j+1)^{{n_j+1}\choose{n_j}}\longrightarrow\cdots\longrightarrow\]
\[\longrightarrow E^\vee{^{{n_j+1}\choose{a_j}}}\longrightarrow \OO\boxtimes\left(\Omega_{\PP^{n_j}}^{a_j}\right)(a_j+1)\otimes E^\vee\longrightarrow0\]
Since \\
\[H^{n_1+\cdots+n_s-1}(E(-n_1-1,\cdots,-n_{s-1}))=\]\\
\[\cdots\cdots\cdots\cdots\cdots\cdots\cdots\cdots\cdots\cdots\cdots\cdots\cdots\cdots\cdots\cdots\cdots\cdots\cdots\]\\
\[H^{n_1+\cdots+n_s-n_j+a_j+1}(E(-n_1-1,\cdots,-n_{s-1}-1,-a_j-3))=\]\\
\[H^{n_1+\cdots+n_s-n_j+a_j}(E(-n_1-1,\cdots,-n_{s-1}-1,-a_j-2))=\]\\
by Serre's duality we have $H^1(E^\vee)=\cdots=$
\[=H^{n_1+\cdots+n_{s-1}-a_j-1}(E^\vee(0,-n_1-\cdots-n_{s-1}+a_j+2))=H^{n_1+\cdots+n_{s-1}+a_j}(0,\cdots,-n_1-\cdots-n_{s-1}+a_j)\]
we have the surjection map $\psi_2$
\[H^0(\OO\boxtimes\Omega_{\PP^{n_j}}^{a_j}(1+a_j)\otimes E^\vee)\]
\[
\begin{CD}
@.@V^{surjects}VV\\
@.H^{n_1+\cdots+n_{s-1}+a_j}(E^\vee(-n_1,\cdots,-n_{s-1}+a_j)\\
\end{CD}
\]

\vspace{1cm}

From these surjections we have the surjective maps

\[H^0(\OO\boxtimes\left(\Omega_{\PP^{n_j}}^{a_j}\right)^\vee(-1-a_j)\otimes E)\otimes H^{n_1+\cdots+n_{s-1}-a_j}(E^\vee(0,-n_1,\cdots,-n_{s-1}+a_j) \]
\[
\begin{CD}
@.@V^{surjects}VV\\
@.H^{n_1+\cdots+a_j+\cdots+n_s\cdots+n_s}(E(-n_1,\cdots,-n_{s-1},-a_j)\otimes H^{n_1+\cdots+n_{s-1}-a_j}(E^\vee(0,-n_1,\cdots,-n_{s-1}+a_j) \\
\end{CD}
\]
 and 

 \[H^{n_1+\cdots+a_j+\cdots+n_s\cdots+n_s}(E(-n_1,\cdots,-n_{s-1},-a_j)\otimes H^0(\OO\boxtimes\Omega_{\PP^{n_j}}^{a_j}(1+a_j)\otimes E^\vee)\]
\[
\begin{CD}
@.@V^{surjects}VV\\
@.H^{n_1+\cdots+a_j+\cdots+n_s\cdots+n_s}(E(-n_1,\cdots,-n_{s-1},-a_j)\otimes H^{n_1+\cdots+n_{s-1}+a_j}(E^\vee(-n_1,\cdots,-n_{s-1}+a_j) \\
\end{CD}
\]

\vspace{1cm}

\noindent Using Yoneda paring we have:

\vspace{1cm}

\[H^0(\OO\boxtimes\left(\Omega_{\PP^{n_j}}^{a_j}\right)^\vee(-1-a_j)\otimes E)\otimes H^0(\OO\boxtimes\Omega_{\PP^{n_j}}^{a_j}(1+a_j)\otimes E^\vee)\]
\[\begin{CD}
@.@VVV\\
@.H^{0}(\OO)\cong \CC \\
\end{CD}
\]

\[H^0(\OO\boxtimes\left(\Omega_{\PP^{n_j}}^{a_j}\right)^\vee(-1-a_j)\otimes E)\otimes H^{n_1+\cdots+n_{s-1}-a_j}(E^\vee(0,-n_1,\cdots,-n_{s-1}+a_j) \]
\[\begin{CD}
@.@VVV\\
@.H^{{n_1+\cdots+n_{s-1}-a_j}}(\OO\boxtimes\Omega_{\PP^{n_j}}^{a_j}(1+a_j))\cong \CC \\
\end{CD}
\]

\[\displaystyle{H^{\sum_{i=1}^sn_i+a_j}(\OO\boxtimes\left(\Omega_{\PP^{n_j}}^{a_j}\right)^\vee(-1-a_j)\otimes E(-n_1-1,\cdots,-n_s))\otimes M}\]
\[\begin{CD}
@.@VVV\\
@.H^{{n_1+\cdots+n_{s}}}(\OO(-n_1-1,\cdots,-n_s-1))\cong \CC \\
\end{CD}
\]
where $M= H^{{\sum_{i=1}^{s-1}n_i-a_j}}(\OO\boxtimes\Omega_{\PP^{n_j}}^{a_j}(1+a_j)\otimes E^\vee(0,-n_1,\cdots,-n_{s-1}+a_j) $\\
\\
We can find non-degenerate pairs of maps
\[g_j:\OO\boxtimes\Omega_{\PP^{n_j}}^{a_j}(1+a_j)\lra E\] 

and

\[f_j:E\lra\OO\boxtimes\Omega_{\PP^{n_j}}^{a_j}(1+a_j)\]

such that $f_i\circ g_i\neq0$ and since each $\Omega_{\PP^{n_j}}^{a_j}(1+a_j)$ is simple for $i.j=1,\ldots,n_{s-1}$ then $E$ is direct summand.
Hence $E$ has one of the following vector bundles as direct summands
\begin{enumerate}
    \renewcommand{\theenumi}{\alph{enumi}}
    \item $\OO$
    \item $\OO(0,\cdots,0,1)$, $\OO(0,\cdots,1,0)$,$\ldots$, $\OO(0,1,\cdots,0,0)$ and $\OO(1,0,\cdots,0,0)$
    \item $\OO\boxtimes\Omega_{\PP^{n_i}}^{a_i}(a_i+1)$ where $1\leq a_i\leq n_i-1$ or
    \item $\Omega_{\PP^{n_i}}^{a_i}(a_i+1)\boxtimes\OO$ where $1\leq a_i\leq n_i-1$ 
\end{enumerate}
That gives the proof of (1)$\Longrightarrow$(2).
\\

\noindent Now to prove (1)$\Longrightarrow$(2):
we need to show that for any $a_j\in\{1,2,\cdots,n_j-1\}$,\\
$\OO\boxtimes\Omega_{\PP^{n_j}}^{a_j}(1+a_j)$ satifies the conditions of (1).\\
For that we consider all the non-zero cohomology groups so that conditions  of (1) hold.
\end{proof}

\newpage

\section{Acknowledgment}
\noindent I wish to express my  thanks to my colleagues at the Department of Mathematics at the University of Nairobi for granting 
me leave in order to pursue my research work. Lastly, I am supremely grateful to Melissa, my wife and our 3 kids Amelia, Jerome and Chuksie who are always supportive of my pursuits.\\

\noindent \textbf{Data Availability statement}
My manuscript has no associate data.

\noindent \textbf{Conflict of interest}
On behalf of all authors, the corresponding author states that there is no conflict of interest.

\end{document}